\documentclass[a4paper,12pt]{amsart} 
\usepackage{amssymb} 
\usepackage{amsmath} 
\usepackage{amsthm}
\usepackage{mathtools}
\usepackage{color}
\usepackage[latin1]{inputenc}
\usepackage{amsfonts}
\usepackage{rotating}
\usepackage{stmaryrd}
\usepackage{comment}
\usepackage{accents}
\usepackage[english]{babel}
\usepackage{mathrsfs}
\usepackage{stackrel}
\usepackage[T1]{fontenc}
\usepackage[all]{xy}
         \newtheorem{definition}{Definition}[section]
	\newtheorem{remark}[definition]{Remark}
	
	\newtheorem{proposition}[definition]{Proposition}

	\newtheorem{theorem}[definition]{Theorem}
	
	\newtheorem{lemma}[definition]{Lemma}

	\newtheorem*{acknowledgement}{Acknowledgements}
		
	\def\Hom{{\rm{Hom}}}

	\def\Aut{{\rm{Aut}}}

	\def\Tr{{\rm{Tr}}}

	\def\Ind{{\rm{Ind}}}
	\def\Res{{\rm{Res}}}

	\def\Br{{\rm{Br}}}

	\def\per{{\rm{per}}}

	\def\Ind{{\rm{Ind}}}
	
	\def\Res{{\rm{Res}}}

	\def\Br{{\rm{Br}}}

	\def\CF{{\rm{CF}}}
	\def\Ho{{\rm{Ho^{b}}}}
	\def\Co{{\rm{Comp^{b}}}}
	\def\per{{\rm{-perm}}}

\title{Splendid and Perverse Equivalences}
\author{L{\'e}o Dreyfus-Schmidt}
\address{Universti\'e Paris Denis Diderot - Institut de Math\'ematiques de Jussieu - Paris Rive Gauche, Bat\^iment Sophie Germain, 75205 Paris Cedex 13, France }
\email{leo.dreyfus-schmidt@imj-prg.fr}
\begin{document}

\maketitle

\begin{abstract}
Inspired by the works of Rickard on splendid equivalences ([Ric96]) and of Chuang and Rouquier on perverse equivalences ([ChRo]), we are here interested in the combination of both, \textit{i.e.} a splendid perverse equivalence. This is naturally the right framework to understand the relations between global and local perverse equivalences between blocks of finite groups, as a splendid equivalence induces local derived equivalences via the Brauer functor. We prove that  under certain conditions, we have an equivalence between a perverse equivalence between the homotopy category of $p$-permutation modules and local derived perverse equivalences, in the case of abelian defect groups.
\end{abstract}

\tableofcontents
	
\setcounter{section}{-1}
\section{Motivation and notations}
We here work on the classical dynamic of global and local properties for the representation theory of a finite group $G$. We recall that by global one means a finite group $G$ while by local one means working at the level of centralizers (or normalizers) of $p$-subgroups of $G$.\\
A splendid complex clearly restricts to an equivalence between the homotopy categories of $p$-permutation modules. Surprisingly, the converse also holds. However if one was to add that the equivalences are moreover perverse, we only have that a perverse equivalence at the level of the homotopy categories of $p$-permutation modules implies a derived perverse equivalence. 
\\We shall see that in order to obtain local perverse derived equivalences, one should start with the stronger condition of a perverse homotopic equivalence on the homotopy category of $p$-permutation modules.\\
As one would hope to realize the global-local connection in a commutative diagram, we shall be working in the homotopy category of $p$-permutation modules rather than in the corresponding derived category. That way, we shall see that from a global perverse homotopy equivalence, one obtains local derived perverse equivalences. In our attempt to go back up from local to global, we will introduce the refined notion of perverse equivalence relative to a partial order. \\
Finally, we will illustrate this by a careful study of the cyclic case, and see that Rouquier's splendid complex ([Rou94]) does not necessarily realize a global perverse equivalence although it always induces locally perverse equivalences. Last but not least, we will make precise the connection between the local perversities of Rouquier's complex and the generalised decomposition numbers of the block.\\
\\The aim of Section 1 is to develop this theme of global versus local along Boltje and Xu's notion of $p$-permutation equivalence. In Section 2 and 3, we will make precise a result of Rickard on splendid complexes connecting a splendid tilting complex $X$ with its image by the Brauer functor $\Br_{\Delta Q}(X)$. In Section 4, we make the connection between perverse equivalences at the level of the centralizer of a $p$-group and at the level of the corresponding normalizer. Then in Section 5, we show that a global perverse homotopy equivalence give rises to local derived perverse equivalence. We then introduce the notion of perverse equivalence relative to a partial order in order to go back up, from the data of local perverse derived equivalence, to a global perverse homotopy equivalence. Then, we gather all of the above results and prove our main result. Finally we study in Section 6 Rouquier's splendid complex for a block with cyclic defect group by computing global and local perversities, and by connecting those to the generalised decomposition numbers attached to the block.

\begin{acknowledgement}
I wish to thank Rapha\"el Rouquier, who introduced me to the theory of perverse equivalences, and for his guidance throughout this research project. I also wish to thank Olivier Dudas for many helpful comments and suggestions. 
 \end{acknowledgement}

\subsection{Notations}
We denote by $(\mathcal{O},K,k)$ a $p$-modular system, \textit{i.e.} $\mathcal{O}$ is a discrete valuation ring, with field of fractions $K$ of characteristic $0$, large enough for all groups considered here, and residue field $k$ of characteristic $p>0$. For $G$ a finite group, we choose to define the diagonal of $G$ as $\Delta G:=\{(g,g^{-1})| g\in G\}$. For $p$ a prime number, the \textit{$p$-core} of $G$ is defined to be the largest normal $p$-subgroup of $G$ and is denoted by $O_{p}(G)$. An element $x\in G$ is said to be \textit{p-regular} if its order is prime to $p$. The set of $p$-regular elements of $G$ is denoted by $G_{reg}$. \\
If $H$ and $K$ are subgroups of $G$, we write $H\subseteq_{G}K$ to say that $H$ is a subgroup of $K$ up to conjugation in $G$.\\
For $A$ a symmetric $R$-algebra ($R$ either $k$ or $\mathcal{O}$), ${A}$\text{-mod} denotes the category of finitely generated $A$-modules. Let $\mathcal{C}$ be an additive category and $\mathcal{A}$ an abelian category. Then $\Co(\mathcal{C})$ denotes the category of bounded complexes of objects of $\mathcal{C}$, $\Ho(\mathcal{C})$ the homotopy category of $\Co(\mathcal{C})$ and $\mathcal{D}^{b}(\mathcal{A})$ denotes the bounded derived category of $\mathcal{A}$.\\
By $K_{0}(RG)$, we denote the Grothendieck group of finitely generated $RG$-modules. The isomorphism classes $[[S]]$ of simple $RG$-module form a $\mathbb{Z}$-basis of $K_{0}(RG)$. The space of class functions of $G$ taking values in $K$ will be denoted by $\CF(G,K)$. Let $e$ and $f$ denote the principal block idempotents of $kG$ and $kH$ respectively. For $Q\leq P$, we denote by $e_{Q}$ (resp. $f_{Q}$) the principal block idempotent of $kC_{G}(Q)$ (resp. $kC_{H}(Q)$).

\subsection{Reminder}
We will now recall some classical definitions. 

\begin{definition}Two finite groups $G$ and $H$ with a common Sylow $p$-subgroup $P$ share the same $p$-local structure if for every $Q_{1}$ and $Q_{2}$ subgroups of $P$ with $\theta:Q_{1}\rightarrow Q_{2}$ an isomorphism, then there is an element $g\in G$, such that $\theta(q)=q^{g}$ for all $q\in Q_{1}$ if and only if there is an element $h\in H$ such that $\theta(q)=q^{h}$ for all $q\in Q_{1}$.
\end{definition} 

For instance, if $H$ is a subgroup of $G$, $P$ is abelian and $H$ contains $N_{G}(P )$ then $H$ and $G$ have the same $p$-local structure. From now on, we will always assume that we are in this particular situation. 

\begin{definition}Let $G$ be a finite group and $Q$ a $p$-subgroup of $G$. We denote by $\Br_{Q}$ the Brauer functor $\Br_{Q}: {kG}\text{-mod} \rightarrow {kN_{G}(Q)}\text{-mod}$ defined for $M$ a $kG$-module by $$\Br_{Q}(M)=M^{Q}/(\sum_{P<Q}\Tr_{P}^{Q}M^{P}),$$ the quotient of $Q$-fixed points of M by the relative traces from all proper subgroups $P$ of $Q$ of the $P$-fixed points. We will also write $M(Q)$ for $\Br_{Q}(M)$.
\end{definition}

We will be mostly interested in a particular type of modules, the \textit{p-permutation modules} (see for example [Bro85]). These are direct summands of permutation modules. In view of Green's theory of vertices and sources, they are also known as \textit{trivial source modules}. We denote by $kG\text{\per}$ the full subcategory of $kG\text{-mod}$ of $p$-permutation modules. In what follows, we will consider the restriction of the Brauer functor $\Br_{Q}:kG\text{\per}\rightarrow kN_{G}(Q)\text{\per}$.\\ 
Note that if $M$ is a permutation module with a $G$-stable basis $X$, we have that $M(Q)\simeq k[X^{Q}]$. Considering the group algebra $kG$ as a $k[G\times G]$-module, $kG(\Delta Q)$ is naturally isomorphic to the group algebra $kC_{G}(Q)$. A key feature of the Brauer construction is the following isomorphism, for $M$ and $N$ $p$-permutation $kG$-modules, not only do we have $$
M(Q)\otimes_{k}N(Q)\stackrel{\sim}{\rightarrow}(M\otimes_{k}N)(Q),$$ but we also have the isomorphism $$
M(Q)\otimes_{kC_{G}(Q)}N(Q)\stackrel{\sim}{\rightarrow}(M\otimes_{kG}N)(Q).$$
Another important feature is that the vertex of an indecomposable $p$-permutation $kG$-module $M$ is precisely the maximal $p$-subgroup of $G$ such that $\Br_{Q}(M)\neq 0$. We now recall the very useful Brou{\'e}-Puig's parametrization of $p$-permutation modules. We put $\bar{N}_{G}(Q)=N_{G}(Q)/Q$

\begin{proposition}([Bro85,Theorem 3.2])
The correspondence $M\mapsto (Q,\Br_{Q}(M))$ for $Q$ a vertex of $M$ defines a bijection from the set of isomorphism classes of indecomposable $p$-permutation $kG$-modules to the set of conjugacy classes of pairs $(Q,N)$, where $Q$ is a $p$-subgroup of $G$ and $N$ is an isomorphism class of an indecomposable projective $k\bar{N}_{G}(Q)$-module.
\end{proposition}

Note that this parametrization is compatible with the decomposition of the group algebras into blocks.
\\In [Bro90], Brou{\'e} introduces the notion of \textit{perfect} character between $\mathcal{O}Ge$ and $\mathcal{O}Hf$ as follows.
\begin{definition}A perfect character is an element $\mu\in K_{0}(KGe,KHf)$ satisfying the following:
\begin{itemize}
\item $\forall g\in G$, $\forall h\in H$, {\large $\frac{\mu(g,h)}{|C_{G}(g)|}$ }$\in \mathcal{O}$ and {\large $\frac{\mu(g,h)}{|C_{H}(h)|}$} $\in \mathcal{O}$.
\item if $\mu(g,h)\neq 0$, then $g$ has order prime to $p$ if and only if $h$ has order prime to $p$.
\end{itemize}
\end{definition}

To $\mu\in K_{0}(KGe,KHf)$, we associate isomorphisms $I_{\mu}:K_{0}(KHf)\rightarrow K_{0}(KGe)$ and $R_{\mu}:K_{0}(KFe)\rightarrow K_{0}(KHf)$ and we say that $I_{\mu}$ is a \textit{perfect} isometry if $\mu$ is a perfect character. If the above is considered as the standard definition of a perfect isometry, we would like to notice that one might prefer (as we do) the equivalent definition of [Bro90, Proposition 4.1]. Kindly, we will set the mind of the anxious reader at rest by pretending that one can think of an isotypy as a ``nicely'' compatible family of perfect isometries.\\
By $T(RG)$, we denote the representation ring of $p$-permutations $RG$-module. Also, we will write $T(RG,RH)$ for $T(RG\otimes_{R}RH^{\text{opp}})$. The isomorphism classes $[M]$ of indecomposable $p$-permutation $RG$-modules form a $\mathbb{Z}$-basis of $T(RG)$. For $Q$ a $p$-subgroup of $G$ we denote by $T^{Q}(RG)$ the subgroup of $T(RG)$ generated by relatively $Q$-projective $p$-permutation modules. For convenience, if $\gamma \in T(RG)$, we will put $\gamma(Q)=\Br_{Q}(\gamma)$. To $\gamma \in T(\mathcal{O}G)$, we associate its character $\mu(\gamma) \in K_{0}(KG)$.\\
The tensor product $-\otimes_{RH} -$ induces a $\mathbb{Z}$-bilinear map 
$$
T(RG,RH)\times T(RH,RL)\rightarrow T(RG,RL),\ (\gamma,\delta)\mapsto \gamma \stackrel[H]{}{\cdot}  \delta,
$$
for any third group $L$. Also, taking the $R$-dual induces an isomorphism 
$$
T(RG,RH)\rightarrow T(RH,RG),\ \gamma \mapsto \gamma^{*}.
$$
We define two types of equivalences between $A$ and $B$, two symmetric $R$-algebras: the so-called \textit{Rickard equivalences} and the \textit{stable equivalences}. We say that an $(A,B)$-bimodule $M$ is \textit{exact} if it is projective as a left $A$-module and as a right $B$-module.

\begin{definition}
A bounded complex $X$ of exact $(A,B)$-bimodules induces a Rickard equivalence if \begin{itemize}
\item$X\otimes_{B}X^{*}\simeq A\oplus Z_{1}$ as complexes of $(A,A)$-bimodules
\item$X^{*}\otimes_{A}X\simeq B\oplus Z_{2}$ as complexes of $(B,B)$-bimodules, \end{itemize}
 where $A$ and $B$ are concentrated in degree $0$, and $Z_{1}$ and $Z_{2}$ are homotopy equivalent to $0$.
\end{definition}

A Rickard complex $X$ then induces a derived equivalence $X\otimes_{B}-:\mathcal{D}^{b}(B\text{-mod}) \stackrel{\sim}{\rightarrow} \mathcal{D}^{b}(A\text{-mod})$.

\begin{definition}
A bounded complex $X$ of exact $(A,B)$-bimodules induces a stable equivalence if \begin{itemize}
\item$X\otimes_{B}X^{*}\simeq A\oplus Z'_{1}$ as complexes of $(A,A)$-bimodules
\item$X^{*}\otimes_{A}X\simeq B\oplus Z'_{2}$ as complexes of $(B,B)$-bimodules, \end{itemize}
 where $A$ and $B$ are concentrated in degree $0$, and $Z'_{1}$ and $Z'_{2}$ are homotopy equivalent to complexes of projectives bimodules.
\end{definition}

Let $RA$ and $RB$ be block algebras of $RG$ and $RH$ respectively. There is a specific type of Rickard complexes between block algebras, called \textit{splendid complexes} introduced by Rickard [Ric96].

\begin{definition}
A complex $X\in \Co({RA}\text{-mod-}{RB})$ is \textit{splendid} if its terms (viewed as $R[G\times H^{\text{opp}}]$-modules) are direct summands of finite direct sums of modules of the form $\Ind_{\Delta Q}^{G\times H^{\text{opp}}}(R)$ for $Q\leq P$ and $X$ realizes a Rickard equivalence between $\mathcal{D}^{b}(RA)$ and $\mathcal{D}^{b}(RB)$.
\end{definition}

It is shown in [Ric96] that a splendid equivalence induces an isotypy at the level of the Grothendieck group. Hence, this might lead us to believe that the derived equivalence predicted by Brou{\'e}, between the derived category of a block with abelian defect and its Brauer correspondent, should be splendid.\\
There is another specific type of derived equivalences, introduced by Chuang and Rouquier (cf. [ChRo]), called \textit{perverse equivalences}. They can be seen as filtered derived equivalences, \textit{i.e.} as a patching of Morita equivalences on each stratum of the filtration.\\
Let $\mathcal{S}$ (resp. $\mathcal{S}'$) be the set of isomorphism classes of simple objects of $RA$ (resp. $RB$).
Consider \begin{itemize}
\item a filtration $\mathcal{S}_{\bullet}=(\emptyset=\mathcal{S}_{-1}\subset \mathcal{S}_{0} \subset \ldots \subset \mathcal{S}_{r}=\mathcal{S})$
\item a filtration $\mathcal{S}'_{\bullet}=(\emptyset=\mathcal{S}'_{-1}\subset \mathcal{S}'_{0} \subset \ldots \subset \mathcal{S}'_{r}=\mathcal{S}')$
\item and a function $p:\{0,\cdots,r \}\rightarrow \mathbb{Z}$.
\end{itemize}

\begin{definition}
An equivalence $F:\mathcal{D}^{b}(RA)\stackrel{\sim}{\rightarrow} \mathcal{D}^{b}(RB)$ is perverse relative to $(\mathcal{S}_{\bullet},\mathcal{S}'_{\bullet},p)$ if the following holds:
\begin{itemize}
\item given $V\in \mathcal{S}_{i}-\mathcal{S}_{i-1}$, then the composition factors of $H^{r}(F(V))$ are in $\mathcal{S}'_{i-1}$ for $r\neq -p(i)$ and there is a filtration $L_{1}\subset L_{2}\subset H^{-p(i)}(F(V))$ such that the composition factors of $L_{1}$ and of $H^{-p(i)}(F(V))/L_{2}$ are in $\mathcal{S}'_{i-1}$ and $L_{2}/L_{1} \in \mathcal{S}'_{i}-\mathcal{S}'_{i-1}$.
\item The map $V\mapsto L_{2}/L_{1}$ induces a bijection $\mathcal{S}_{i}-\mathcal{S}_{i-1}\stackrel{\sim}{\rightarrow} \mathcal{S}'_{i}-\mathcal{S}'_{i-1}$.
\end{itemize}
\end{definition}

Note that if $F$ is a perverse equivalence with $p=0$, then $F$ restricts to a Morita equivalence $RA\text{-mod}\stackrel{\sim}{\rightarrow} RB\text{-mod}$.\\
In the context of $\mathfrak{sl}_{2}$-categorification defined by the previous two authors in [ChRo08], the Rickard complex $\Theta$ that gives a self-derived equivalence $\Theta:\mathcal{D}^{b}(\mathcal{A}) \stackrel{\sim}{\rightarrow} \mathcal{D}^{b}(\mathcal{A})$ for $\mathcal{A}$ a $\mathfrak{sl}_{2}$-categorification is perverse. Another famous example of a perverse equivalence is the complex of Rickard and Cabanes ([CaRi01]) which gives the Alvis-Curtis duality at the level of Grothendieck groups. Note also that the splendid complex $X$ of the Section $3$ of [Ric96] between the principal block of $kA_{5}$ and $kA_{4}$ gives a perverse derived equivalence. This is a prototype of the so-called \textit{elementary} perverse equivalences.\\
In Section 5, we will define another type of perverse equivalence, at the level of the homotopy category of additive categories.

\section{The classical dynamic of global and local}
Within this paradigm of connecting global and local properties, a first natural question is: what can be said homologically with no binds of perversity? More precisely, if $X$ is a complex of relatively $\Delta P$-projective $p$-permutations $(kG, kH)$-bimodules which locally induces derived equivalences, \text{i.e.} at the level of centraliser of $p$-element, does $X$ realize a derived equivalence between $kG$ and $kH$? The answer is not exactly. Indeed, thanks to a result of Bouc-Rouquier [Rou01, Theorem 5.6] one can only hope for a global stable equivalence:

\begin{proposition}([Rou01, Theorem 5.6])
Let $X\in \Co({kGe}\text{-mod-}{kHf})$ be a complex of relatively $\Delta P$-projective $p$-permutations $(kG, kH)$-bimodules. The following assertions are equivalent:
\begin{enumerate}
\item $X$ induces a stable equivalence between $kGe$ and $kHf$. 
\item For every non-trivial subgroup $Q\leq P$, the complex $\Br_{\Delta Q}(X)$ induces a Rickard equivalence between $kC_{G}(Q)e_{Q}$ and $kC_{H}(Q)f_{Q}$.
\item For every subgroup $Q$ of order $p$ in $P$, the complex $\Br_{\Delta Q}(X)$ induces a Rickard equivalence between $kC_{G}(Q)e_{Q}$ and $kC_{H}(Q)f_{Q}$.
\end{enumerate}
\end{proposition}

Boltje and Xu introduced an intermediate notion, that lies between a splendid equivalence and an isotypy, the so-called $p$-permutation equivalence (cf. [BoXu08]). Indeed, they showed that, not only a splendid equivalence induces a $p$-permutation equivalence but also that a $p$-permutation equivalence induces an isotypy. 

\begin{definition} A $p$-permutation equivalence between $\mathcal{O}Ge$ and $\mathcal{O}Hf$ is an element $\gamma \in T^{\Delta P}(\mathcal{O}Ge,\mathcal{O}Hf)$ satisfying
$$\gamma \stackrel[H]{}{\cdot} \gamma^{*}=[\mathcal{O}Ge] \in T(\mathcal{O}Ge,\mathcal{O}Ge)$$
and
$$\gamma^{*} \stackrel[G]{}{\cdot} \gamma=[\mathcal{O}Hf] \in T(\mathcal{O}Hf,\mathcal{O}Hf).$$
\end{definition}

\begin{lemma}
Let $\gamma \in T^{\Delta P}(\mathcal{O}Ge,\mathcal{O}Hf)$ with $\mu(\gamma)$ an isometry. Then $\gamma$ is a $p$-permutation equivalence between $\mathcal{O}Ge$ and $\mathcal{O}Hf$ if and only if for every $p$-subgroup $Q\neq 1$, $\gamma(\Delta Q)$ is a $p$-permutation equivalence between $\mathcal{O}C_{G}(Q)e_{Q}$ and $\mathcal{O}C_{H}(Q)f_{Q}$.
\end{lemma}
\begin{proof}
One direction is straightforward, as we have $\Br_{\Delta Q}(\gamma \stackrel[H]{}{\cdot} \gamma^{*})\simeq \Br_{\Delta Q}(\gamma) \stackrel[C_{H}(Q)]{}{\cdot} \Br_{\Delta Q}(\gamma^{*})$. We proceed to prove the other direction, from local to global. We write $\tilde{\gamma}:=\gamma \stackrel[H]{}{\cdot} \gamma^{*}-[\mathcal{O}Ge] =[M]-[N]$ where $M$ and $N$ are $p$-permutation $\mathcal{O}Ge$-bimodules and we prove that if for every $p$-subgroup $Q\neq 1$, $\tilde{\gamma}(\Delta Q)=0$ then $\tilde{\gamma}=0$.
We have $[M(\Delta Q)]=[N(\Delta Q)]$ for every $Q$. By Brou{\'e}-Puig's parametrization of permutation modules (cf. [Bro85, Theorem 3.2]) and the Krull-Remak-Schmidt theorem, we have $[M] + [L']=[N] + [L]$, where $L$ and $L'$ are projective $\mathcal{O}Ge$-bimodules. So that $\gamma \stackrel[H]{}{\cdot} \gamma^{*}-[\mathcal{O}Ge] =[L]-[L']$. But if we take the associated character over $K$, and as $\mu \stackrel[H]{}{\cdot} \mu^{*}=[[KGe]]$, we have $[[L]]=[[L']]$. However by injectivity of the Cartan homomorphism (cf. [Ser78, Chapter 16]) $c:K_{0}(kG\text{-proj})\hookrightarrow K_{0}(kG)$, we conclude that $[L]=[L']$ and hence $\gamma \stackrel[H]{}{\cdot} \gamma^{*}=[\mathcal{O}Ge] \in T(\mathcal{O}Ge,\mathcal{O}Ge)$.\end{proof}

The next stronger result relies on the same technique and gives us a partial converse to Boltje and Xu's theorem [BoXu08, Theorem 1.11].

\begin{proposition}
Let $\gamma \in T^{\Delta P}(\mathcal{O}Ge,\mathcal{O}Hf)$. If $(\mu(\gamma({\Delta Q}))_{Q\leq P})$ is an isotypy, then $\gamma$ is a $p$-permutation equivalence.
\end{proposition}
\begin{proof}
In fact, we do not actually need to have an isotypy but only a family of perfect isometries. We prove that if for every $p$-subgroup $Q$, $\mu(\Delta Q)=0$ then $\gamma(\Delta Q)=0$, the rest will follow according to the proof of the previous lemma. We proceed by decreasing induction as follows. First for the Sylow $P$, $\tilde{\gamma}(\Delta P)$ is constituted of $k[N_{G\times H^{\text{opp}}}(\Delta P)/\Delta P]$-projective module and so as $\mu_{\tilde{\gamma}(\Delta P)}=0$ we use again the injectivity of the Cartan homomorphism to conclude that $\tilde{\gamma}(\Delta P)=0$. Now suppose we have proved the above property for any $Q<R$ and let us now prove it for $R$. We mimick the proof of the previous lemma to $\gamma(R)$ as $\gamma(\Delta R)(\Delta Q)=\gamma(\Delta Q)=0$ for every $Q\triangleleft R$. Hence, we can finally conclude that $\gamma$ is indeed a $p$-permutation equivalence.
\end{proof}

\begin{remark}
According to our proof, it is enough to require the local property only for every subgroup $Q<P$ of order $p$. This way, our result has a similar flavor as Proposition 1.1.  
\end{remark}

\section{A commutative diagram}
Here our goal is to understand splendid equivalences locally, \textit{i.e.} at the level of centralizers of $p$-elements, following the fundamental article of Rickard [Ric96]. If a splendid complex give rises to local derived equivalences, our desire as algebraists to realize it in a commutative diagram cannot then be assuaged. Indeed, the Brauer functor cannot be defined on derived categories, as it is neither left or right exact. This suggests that the right framework of study might be the homotopy category of $p$-permutations modules.

\begin{lemma}
For $X$ a $\Delta P$-projective $k[G\times H^{\text{opp}}]$-module and $Y$ a $kH$-module, we have that $\Res^{G\times H^{\text{opp}}\times H}_{G\times \Delta H^{\text{opp}}}(X\otimes_{k}Y)$ is $\Delta P$-projective, where $\Delta P$ is canonically embedded into $G\times \Delta H^{\text{opp}}$.
\end{lemma}	

\begin{proof}	
By assumption, $X\vert \Ind_{\Delta P}^{G\times H^{\text{opp}}}X'$ for $X'$ a $k\Delta P$-module and hence we can write $$\Res_{G\times \Delta H^{\text{opp}}}(\Ind_{\Delta P}^{G\times H^{\text{opp}}}X'\otimes_{k} Y)=\Res_{G\times \Delta H^{\text{opp}}}(\Ind_{\Delta P\times H^{\text{opp}}}^{G\times H^{\text{opp}}\times H}(X'\otimes_{k} Y)).$$ By Mackey's theorem, the latter is equal to:
$$\bigoplus_{x\in \Delta P\times H\backslash (G\times H^{\text{opp}}\times H)/G\times \Delta H^{\text{opp}}}\Ind_{(G\times \Delta H^{\text{opp}})\cap (\Delta P \times H^{\text{opp}})^{x}}^{G\times \Delta H^{\text{opp}}} x^{*}(X'\otimes_{k} Y).$$
However, $\Delta P\times H\backslash (G\times H^{\text{opp}}\times H)/G\times \Delta H^{\text{opp}}=\{1\}$ and so there is only one term in the previous direct sum.
\\Hence as $(G\times H)$-module, $\Res_{G\times \Delta H^{\text{opp}}}(X\otimes_{k} Y)$ is $\Delta P$-projective.
\end{proof}
We can now state the following result which is taken from [Ric96].

\begin{proposition}[{\`a} la Rickard]
Let $X$ be a complex whose terms are relatively $\Delta P$-projective p-permutation $kG$-$kH$-bimodules, then for every subgroup $Q\leq P$, we have a commutative diagram:
$$\xymatrix{\Ho(kH\per)\ar[r]^{X\otimes_{kH}-} \ar[d]^{\Br_{Q}}  & \Ho(kG\per)\ar[d]^{\Br_{Q}} \\ \Ho(kC_{H}(Q)\per) \ar[r]^{X_{Q}\otimes_{kC_{H}(Q)}-} & \Ho(kC_{G}(Q)\per)} $$
\end{proposition}

Here $X_{Q}:=\Br_{\Delta Q}(X)$.

\begin{proof}
Everything can be carried on termwise: we first apply the previous lemma and we notice that if $M$ and $N$ are $p$-permutations $kG$-modules, then so is $M\otimes_{k}N$. \\Hence for $Y$ a $p$-permutation $kH$-module, we have that the terms of $X\otimes_{k}Y$ are $\Delta P$-projective $p$-permutation $k[G\times H]$-modules. We can then apply an analogous  version of the Lemma 4.2 and Lemma 4.3 of [Ric96] on tensor products rather than on Hom-spaces, in order to find that $\Br_{Q}(X\otimes_{kH}Y)\simeq \Br_{\Delta Q}(X)\otimes_{kC_{H}(Q)}\Br_{Q}(Y)$. 
\end{proof}

However, we should have been more careful here as the Brauer functor is actually defined from $\Ho(kH\per)$ to $\Ho(kN_{H}(Q)\per)$ (or even more precisely $\Ho(k(N_{H}(Q)/Q)\per)$), and so we need to clarify things here. More concretely, we have to be careful when going from the normaliser down to the centraliser. This is the purpose of the following section.

\section{A white lie and another commutative diagram}
We shall start here by setting some notations. Firstly, recall that $e$ and $f$ are principal block idempotents and $P$ is an abelian $p$-Sylow. Since as $P$ is abelian, we have $p\nmid[N_{G}(Q):C_{G}(Q)]$ for any $p$-subgroup $Q$.\\
If $X$ is a splendid tilting complex of $kGe$-$kHf$-bimodules, we denote by $X_{Q}:=\Br_{\Delta Q}(X)$, the corresponding complex of $kN_{G\times H^{\text{opp}}}(\Delta Q)$-module. By restriction to $C_{G}(Q)\times C_{H^{\text{opp}}}(Q)$, $X_{Q}$ still gives a Rickard equivalence.\\
More interestingly, according to a lemma of Marcus ([Ma96]), we can lift this Rickard complex so that $X'_{Q}:=\Ind_{N_{G\times H^{\text{opp}}}(\Delta Q)}^{N_{G}(Q)\times N_{H}(Q)^{\text{opp}}}(X_{Q})$ is also a Rickard equivalence between $kN_{G}(Q)e_{Q}$ and $kN_{H}(Q)f_{Q}$. From now on, we will write $N:=N_{G\times H^{\text{opp}}}(\Delta Q)$.\\
\\In fact, we claim that we can link together the Rickard equivalences for the centralizers with the Rickard equivalences for the normalizers.

\begin{proposition} 
Let $X$ be a splendid tilting complex of $kGe$-$kHf$-bimodules, then for every subgroup $Q\leq P$, we have the commutative diagram:
$$\xymatrix{\Ho(kHf\per)\ar[r]^{X\otimes_{kH}-} \ar[d]^{\Br_{Q}}  & \Ho(kGe\per)\ar[d]^{\Br_{Q}} \\ \Ho(kN_{H}(Q)f_{Q}\per) \ar[r]^{X'_{Q}\otimes_{kN_{H}(Q)}-}\ar[d]^{\Res_{C_{H}(Q)}} & \Ho(kN_{G}(Q)e_{Q}\per)\ar[d]^{\Res_{C_{G}(Q)}} \\ \Ho(kC_{H}(Q)f_{Q}\per)\ar[r]^{\tilde{X}_{Q}\otimes_{kC_{H}(Q)}-} & \Ho(kC_{G}(Q)e_{Q}\per)}$$
where for the sake of simplicity, we set $\tilde{X}_{Q}:=\Res_{C_{G}(Q)\times C_{H}(Q)^{\text{opp}}}(X_{Q})$.
\end{proposition}

 \begin{proof}
 The bigger square is in fact the right formulation of the previous proposition \textit{{\`a} la Rickard}, we have a canonical isomorphism of functors:
\begin{eqnarray}
\tilde{X}_{Q}\otimes_{kC_{H}(Q)}\Res_{C_{H}(Q)}(\Br_{Q}(-))\stackrel{\sim}{\rightarrow} \Res_{C_{G}(Q)}(\Br_{\Delta Q}(X\otimes_{kH}-))
 \end{eqnarray}
The commutativity of the bottom square can be expressed as $$\Res_{C_{G}(Q)}(\Ind_{N}^{N_{G}(Q)\times N_{H}(Q)^{\text{opp}}}X_{Q}\otimes_{kN_{H}(Q)}-)\stackrel{\sim}{\rightarrow} \tilde{X}_{Q}\otimes_{kC_{H}(Q)}\Res_{C_{H}(Q)}(-)$$
Applying Mackey's formula and as $1\times N_{H}(Q)^{\text{opp}}\backslash N_{G}(Q)\times N_{H}(Q)^{\text{opp}}/N=1$, we get:
\begin{eqnarray}
\Res_{C_{G}(Q)}(\Ind_{N}^{N_{G}(Q)\times N_{H}(Q)^{\text{opp}}}X_{Q})\stackrel{\sim}{\rightarrow}  \Ind_{C_{G}(Q)\times C_{H}(Q)^{\text{opp}}}^{C_{G}(Q)\times N_{H}(Q)^{\text{opp}}}X_{Q} 
\end{eqnarray}
 So that the previous expression now becomes :
\begin{eqnarray}
\Ind_{C_{G}(Q)\times C_{H}(Q)^{\text{opp}}}^{C_{G}(Q)\times N_{H}(Q)^{\text{opp}}}X_{Q}\otimes_{kN_{H}(Q)}- \stackrel{\sim}{\leftarrow}  \tilde{X}_{Q}\otimes_{kC_{H}(Q)}\Res_{C_{H}(Q)}(-)
\end{eqnarray}
Now, this is just expressing the adjunction between induction and restriction as an isomorphism between $k$-vector space. Actually this is also an isomorphism of $kC_{G}(Q)$-modules.  
\\It remains to prove that the upper square is commutative. In order to do so, it is enough to prove that both $(1)$ and $(3)$ are $N_{G}(Q)$-isomorphisms.\\\\
For the rest of this proof, we will work with Hom's spaces rather than tensor product, as we would rather deal with fixed than cofixed point. Of course, thanks to the adjunction between Hom-functor and tensor functor, this does not change anything.\\
First of all, we shall point out that it is in no way easy to see that for $Y\in \Ho(kHf\per),$ $N_{G}(Q)$ acts on $\Hom_{kC_{H}(Q)}(X_{Q},Y_{Q})$. What follows is a particularly nice trick to unveil the action. We denote by $Z:=\Hom_{k}(X_{Q},Y_{Q})$ the $(N^{\text{opp}}\times N_{H}(Q))$-bimodule and we try to provide $Z^{\Delta C_{H}(Q)}=\Hom_{kC_{H}(Q)}(X_{Q},Y_{Q})$ with an action of $N_{G}(Q)$. 
If one can find $N'\subseteq N^{\text{opp}}\times N_{H}(Q)$, such that the following short exact sequence holds: 
$$\xymatrix{1\ar[r] & \Delta C_{H}(Q)\ar[r] & N' \ar[r] & N_{G}(Q) \ar[r]& 1}$$ then we could endow $(\Res_{N'}Z)^{\Delta C_{H}(Q)}$ with a \textit{natural} action of $N_{G}(Q)$. We will proceed in two steps, the first one defines $\tilde{N}$ as follows:
$$\xymatrix{N^{\text{opp}}\times N_{H}(Q)\ar[r] & \Aut(\Delta Q)^{\text{opp}}\times \Aut(Q)\ar@{<->}[d] \\ &\Aut(Q)^{\text{opp}}\times \Aut (Q)\\ \tilde{N}\ar[r] \ar@{^(->}[uu] & \Delta(\Aut(Q))\ar@{^(->}[u]  }$$
We invite the attentive reader to check that $N'=\tilde{N}\cap (N_{G}(Q)\times \Delta N_{H}(Q))$ suits us.
Concretely, the action of $g\in N_{G}(Q)$ on $\Hom_{kC_{H}(Q)}(X_{Q},Y_{Q})$ is defined as the action of $(g,h,h) \in N'$: for $\beta \in \Hom_{kC_{H}(Q)}(X_{Q},Y_{Q})$ and $g\in N_{G}(Q)$, $g.\beta:=\beta(g\cdot h)h^{-1}$. \\
We now turn to see that the adjunction $(3)$ is also a $kN_{G}(Q)$-morphism. The Mackey's isomorphism $(2)$ provides $\Ind_{C_{H}(Q)^{\text{opp}}}^{N_{H}(Q)^{\text{opp}}}X_{Q}$ with an action of $N'$:
\begin{eqnarray*}
 k(N_{G}(Q)\times N_{H}(Q)^{\text{opp}})\otimes_{kN}X_{Q} & \stackrel{\sim}{\longrightarrow}  &X_{Q}\otimes_{kC_{H}(Q)}kN_{H}(Q) \\
(g\otimes h)\otimes x & \stackrel{\phi}{\longmapsto} & \big( (g\otimes l).x \big) \otimes l^{-1}h\\
(1\otimes n)\otimes x & \longmapsfrom & x\otimes n
\end{eqnarray*}
where $l\in N_{H}(Q)$ and $g\in N_{G}(Q)$ induce the same automorphism of $Q$, so that $g\otimes h=\underbrace{(g\otimes l)}_{\in N}(1\otimes l^{-1}h)$.\\
That way we define an action of $N_{G}(Q)\times N_{H}(Q)$ on $\Ind_{C_{H}(Q)^{\text{opp}}}^{N_{H}(Q)^{\text{opp}}}(X_{Q})$. For $(g,h)\in N_{G}(Q)\times N_{H}(Q)$ and $x\otimes n \in \Ind_{C_{H}(Q)^{\text{opp}}}^{N_{H}(Q)^{\text{opp}}}(X_{Q})$, we have $(g,h).(x\otimes n):=\phi((g,h).\phi^{-1}(x\otimes n))=\big((g\otimes l).x\big)\otimes l^{-1}nh$ for $k$ as above.\\
Hence for $g\in N_{G}(Q)$ such that $(g,h,h)\in N'$, $g.(x\otimes n)=\big((g\otimes h).x\big)\otimes h^{-1}nh$.\\
Now recall that we had the adjunction $(3)$: 
\begin{eqnarray*}
\Hom_{kC_{H}(Q)}(X_{Q},\Res_{C_{H}(Q)}^{N_{H}(Q)}(Y_{Q}))& \stackrel{\sim}{\rightarrow} & \Hom_{kN_{H}(Q)}(\Ind_{C_{H}(Q)^{\text{opp}}}^{N_{H}(Q)^{\text{opp}}}(X_{Q}),Y_{Q}))\\
\beta & \stackrel{F}\mapsto & \big(x\otimes n \mapsto \beta(x)n^{-1}\big)
\end{eqnarray*}
For $g\in N_{G}(Q)$, $h\in N_{H}(Q)$ such that $(g,h,h)\in N'$, $F(g.\beta)=\big(x\otimes n\mapsto \beta(gxh)h^{-1}n^{-1}\big))=\big(x\otimes n \mapsto \beta(gxh)(h^{-1}n^{-1}h)h^{-1}\big)=g.F(\beta)$.\\
We conclude, as promised, that $(3)$ is an $N_{G}(Q)$-isomorphism. 
\\Finally, a little diagram chasing leads us to our desired conclusion: the upper square is also commutative.
\end{proof}

\section{Perverse Equivalences and Clifford Theory}
Let $G$ be a finite group and $H$ a normal subgroup of $G$ of index prime to $p$ with $G=H\rtimes L$. Let $\mathcal{S}_{G}$ (resp. $\mathcal{S}_{H}$) be the set of isomorphism classes of simple $kG$-module (resp. $kH$-module). We define an equivalence relation on $\mathcal{S}_{G}$ by $M\sim N$ if $\Hom_{kH}(\Res_{H}(M),\Res_{H}(N))\neq 0$.

\begin{lemma}
Induction and restriction gives a bijection $\mathcal{S}_{H}/L \stackrel{\sim}{\rightarrow} \mathcal{S}_{G}/\sim$.
\end{lemma}

Consider $\mathcal{S}_{\bullet}:=(\emptyset=\mathcal{S}_{0}\subset S_{1}\subset ...\subset \mathcal{S}_{r}=\mathcal{S}_{G})$ a filtration  of $\mathcal{S}_{G}$ and a perversity function $p:\{1,...,r\}\rightarrow \mathbb{Z}$. This perversity datum $(p,\mathcal{S}_{\bullet})$ is said to be \textit{H-compatible} if it is compatible with $\sim$.\\
We adapt to our situation a result of [CrRo10]:

\begin{proposition}
Let $X$ be a complex of $kN_{G\times H^{\text{opp}}}(\Delta Q)$-module and let $(p,\mathcal{S}_{C_{G}(Q),\bullet},\mathcal{S}_{C_{H}(Q),\bullet})$ be a $H$-invariant perversity datum with corresponding $H$-compatible datum $(p',\mathcal{S}_{N_{G}(Q),\bullet},\mathcal{S}_{N_{H}(Q),\bullet})$. 
\\Then $\Res_{C_{G}(Q)\times C_{H}(Q)^{\text{opp}}}(X)$ induces a perverse equivalence relative to the datum $(p,\mathcal{S}_{C_{G}(Q),\bullet},\mathcal{S}_{C_{H}(Q),\bullet})$ if and only if $\Ind^{N_{G}(Q)\times N_{H}(Q)^{\text{opp}}}(X)$ induces a perverse equivalence relative to $(p',\mathcal{S}_{N_{G}(Q),\bullet},\mathcal{S}_{N_{H}(Q),\bullet})$.
\end{proposition}

\begin{proof}
The equivalence part is given by Marcus' lemma. For the perversity, we have to use the previous lemma. We refer to [CrRo10] for a more detailed proof.
\end{proof}

\section{About splendid perverse equivalences}
Splendid equivalence gives rise to derived equivalences at the level of centralizers of $p$-elements. It then seems natural to wonder if a splendid perverse equivalence would also give perverse equivalences locally. Under some mild assumptions, the answer is positive.\\
Firstly, we recall the definition of a perverse equivalence for the homotopy category [ChRo].\\
Let $\mathcal{C}, \mathcal{C}'$ be additive categories satisfying the Krull-Schmidt property. We endow them with a structure of exact category via the split exact sequences. Given $C\in \Co(\mathcal{C})$, we denote by $C_{min}\in \Co(\mathcal{C})$ the complex, unique up to isomorphism such that $C\simeq C_{min}$ in $\Ho(\mathcal{C})$ and has no non-zero direct summand that is homotopy equivalent to $0$.\\
Let $I$ be the set of isomorphism classes of indecomposable objects of $\mathcal{C}$. We have a bijective correspondence $\mathcal{I}\mapsto [\mathcal{I}]$ from Serre subcategories of $\mathcal{C}$ to subsets of $I$. Finally, we denote by $\mathcal{I}_{\bullet}$ (resp. $\mathcal{I}'_{\bullet}$) a filtration of $\mathcal{C}$ (resp. $\mathcal{C}'$) of length $r$ by Serre subcategories and consider $p:\{1,...,r\}\rightarrow \mathbb{Z}$.

\begin{definition}
An equivalence $F:\Ho(\mathcal{C})\stackrel{\sim}{\rightarrow} \Ho(\mathcal{C}')$ is perverse relative to $(\mathcal{C}_{\bullet},\mathcal{C'}_{\bullet},p)$ if and only if
\begin{itemize}
\item for $M\in [\mathcal{I}_{i}]-[\mathcal{I}_{i-1}]$, we have $(F(M)_{min})^{r}\in \mathcal{I}'_{i-1}$ for $r\neq -p(i)$ and $(F(M)_{min})^{-p(i)}=M'\oplus L$ for some $M'\in [\mathcal{I}'_{i}]-[\mathcal{I}'_{i-1}]$ and $L\in \mathcal{I}'_{i-1}$.
\item The map $M\mapsto M'$ gives a bijection $[\mathcal{I}_{i}]-[\mathcal{I}_{i-1}]  \stackrel{\sim}{\rightarrow}  [\mathcal{I}'_{i}]-[\mathcal{I}'_{i-1}]$.
\end{itemize}
\end{definition}

The following lemma (cf. [ChRo]) establishes the connection between perverse equivalence for the homotopy category of projective $kG$-modules and perverse equivalence for the derived category of $kG\text{-mod}$. Consider $\mathcal{S}$ and $\mathcal{S}'$ filtrations of $kG\text{-mod}$ and $kH\text{-mod}$ of length $r$. We denote by $\mathcal{P}_{i}$ the additive full subcategory of $kG\text{-proj}$ generated by the projective covers of $V$, $V\in \mathcal{S}-\mathcal{S}_{r-i}$. We define $\bar{p}$ by $\bar{p}(i)=p(r-i+1)$.
\begin{lemma}
Consider an equivalence $F:\mathcal{D}^{b}(kG) \stackrel{\sim}{\rightarrow} \mathcal{D}^{b}(kH)$ that restricts to an equivalence $\bar{F} : \Ho(kG\text{-proj}) \stackrel{\sim}{\rightarrow} \Ho(kH\text{-proj})$.\\
The equivalence $F$ is perverse relative to $(\mathcal{S}_{\bullet},\mathcal{S}'_{\bullet},p)$ if and only if $
\bar{F}$ is perverse relative to $(\mathcal{P}_{\bullet},\mathcal{P}'_{\bullet},\bar{p})$.
\end{lemma}

\subsection{From global to local perversities}
We now have all the tools we need to answer our original question. We consider $X$ a splendid Rickard complex of $kGe$-$kHf$-bimodules that restricts to a perverse equivalence between the homotopy category of $p$-permutation modules. This is an important assumption as this is stronger than asking for a perverse equivalence between derived categories. \\We denote by $\mathcal{I}_{\bullet}$ a filtration of $I$, the set of isomorphism classes of indecomposable $p$-permutation $kHf$-modules. For $Q$ a $p$-subgroup of $H$, we denote by $\mathcal{I}_{Q\bullet}$ the corresponding filtration consisting only of permutation modules of vertex $Q$. Then $\Br_{ Q}(\mathcal{I}_{Q\bullet})$ gives us a filtration on $kN_{H}(Q)f_{Q}\text{-proj}$. Let us see if the induced equivalence $\Ho(kN_{H}(Q)f_{Q}\text{-proj})  \stackrel{\sim}{\rightarrow} \Ho(kN_{G}(Q)e_{Q}\text{-proj}) $ is perverse. For that, we consider the following commutative diagram.
$$\xymatrix{\Ho(kHf\per)\ar[r]^{X\otimes_{kH}} \ar[d]^{\Br_{ Q}}  & \Ho(kGe\per)\ar[d]^{\Br_{ Q}} \\ \Ho(kN_{H}(Q)f_{Q}\per) \ar[r] & \Ho(kN_{G}(Q)e_{Q}\per) \\ \Ho(kN_{H}(Q)f_{Q}\text{-proj})\ar[r]\ar@{^(->}[u] & \Ho(kN_{G}(Q)e_{Q}\text{-proj})\ar@{^(->}[u]}$$
Collecting all of the above, we have the following statements.

\begin{lemma}
For $M$ an indecomposable $p$-permutation $kHf$-module of vertex $Q$, the terms of $X\otimes_{kH}M$ have, up to conjucacy, vertices smaller than $Q$.  
\end{lemma}

\begin{proof}
Let $M$ an indecomposable $p$-permutation $kHf$-module of vertex $Q$ with $M\in [\mathcal{I}_{i}]-[\mathcal{I}_{i-1}]$. Then $\Br_{ Q}(M)$ is a $kN_{H}(Q)f_{Q}$-projective indecomposable. Reciprocally, we know that any $kN_{H}(Q)f_{Q}$-projective indecomposable comes us this way. Now a quick diagram chasing gives us the desired result.
\end{proof}

\begin{proposition}
Let $X$ be a splendid $kGe$-$kHf$-complex that restricts to a perverse equivalence $\Ho(kHf\per)  \stackrel{\sim}{\rightarrow} \Ho(kGe\per)$. Then for any $p$-subgroup $Q$ of $P$, $\Br_{\Delta Q}(X)$ induces a perverse equivalence $\mathcal{D}^{b}(kN_{H}(Q)f_{Q})\stackrel{\sim}{\rightarrow} \mathcal{D}^{b}(kN_{G}(Q)e_{Q})$.
\end{proposition}

\begin{proof} We have to show that we have a perverse homotopy equivalence between the corresponding homotopy category of projective modules thanks to Lemma $5.2$. The only thing we need to check is that the bijections $M\mapsto M'$ (as in Definition 5.1) between each stratum, given by the global perverse equivalence, respect the vertex. A projective module is sent by a splendid complex to a perfect complex and hence if $M$ has vertex $0$, so does $M'$. We can now proceed by induction. With the previous lemma, we know that if $M$ has vertex $Q$, then the terms of $X\otimes_{kH}M$ has vertex smaller than $Q$. However for every vertex $Q'<Q$ we have, up to isomorphism, as many $kG$-permutation modules with vertex $Q'$ than $kH$-permutation modules with vertex $Q'$. As the assignment $M\mapsto M'$ is already a bijection, we conclude by induction that $M'$ is also of vertex $Q$.
\end{proof}

According to Brou{\'e}-Puig's parametrization, we have $I  \stackrel{\sim}{\rightarrow} \{(Q,N_{H}(Q$)-simples)$ \}_{Q}$. The induced perversity datum for elements of vertex $Q$ is denoted by $(\mathcal{I}_{Q\bullet},p_{Q})$, where $p_{Q}$ is the restriction of $p$ to $\mathcal{I}_{Q\bullet}$. We say it is \textit{$C_{H}(Q)$-compatible} if for any $S,S'$ simple $kN_{H}(Q)$-modules such that $\Res_{C_{H(Q)}}(S)\simeq\Res_{C_{H(Q)}}(S')$, then $p(S)=p(S')$.

\begin{definition}
If for all $p$-subgroups $Q$, the induced perversity data $(\mathcal{I}_{Q\bullet},p_{Q})$ are $C_{H}(Q)$-compatible, we say that $(\mathcal{I}_{\bullet},p)$ is locally compatible.
\end{definition}

Now, using Proposition $4.2$ and the previous proposition, we can state the following:

\begin{proposition} 
Let $X$ be a splendid $kGe$-$kHf$-complex that restricts to a perverse equivalence $\Ho(kHf\per)  \stackrel{\sim}{\rightarrow} \Ho(kGe\per)$ with locally compatible perversity datum. Then for all $Q$, $X$ induces perverse equivalences $\mathcal{D}^{b}(kC_{H}(Q)f_{Q})\stackrel{\sim}{\rightarrow} \mathcal{D}^{b}(kC_{G}(Q)e_{Q})$. 
\end{proposition}

\subsection{From local to global?}
We would now want a converse to Proposition $5.4$, so we need to see how, from the data of local perverse equivalence, we could obtain a global perverse equivalence.
First, we need to slightly extend the notion of perverse equivalence for a partial order. Let $\mathcal{C},\mathcal{C}'$ additive categories and $I$ (resp. $I'$) the set of indecomposable objects of $\mathcal{C}$ (resp. $\mathcal{C}'$). We consider $\leq$ (resp. $\leq'$) partial order on $I$ (resp. $I'$) such that $I$ and $I'$ are isomorphic as posets through a map $\phi$. A perversity function is then $p:I \rightarrow \mathbb{Z}$. 

\begin{definition}
An equivalence $F:\Ho(\mathcal{C})\stackrel{\sim}{\rightarrow} \Ho(\mathcal{C}')$ is perverse relative to $(\leq,\leq',p)$ if and only if for $M\in I$, we have $(F(M)_{min})^{r}\in \mathcal{I}'_{<' \phi(M)}$ for $r\neq -p(M)$ and $(F(M)_{min})^{-p(M)}=\phi(M)\oplus L$ for $L\in \mathcal{I}'_{<'\phi(M)}$.
\end{definition}

Of course, if the order is total, we find the original definition of perverse equivalence on homotopy categories.\\
Let us go back to our classical settings and consider $\leq$ a partial order on $\{Q,N_{H}(Q)f_{Q}\text{-simples}\}_{Q}$, or equivalently on $I$, the set of isomorphims classes of indecomposable $p$-permutation $kHf$-module. The next lemma tells us that we can refine any partial order in a particularly interesting way, so that it respects the inclusion of vertices.

\begin{lemma}
Consider an equivalence $X\otimes_{kH}-:\Ho(kHf\per)  \stackrel{\sim}{\rightarrow} \Ho(kGe\per)$ given by a splendid complex $X$. Suppose that $X\otimes_{kH}-$ is perverse relative to $(\leq, p)$. 
\\Then it is perverse relative to $(\leq^{+},p)$ where $\forall Q,P$ $p$-subgroups of $G$, and $M\in N_{H}(Q)f_{Q}$-simple, $N\in N_{H}(P)f_{Q}$-simple, $(Q,M)\leq^{+}(P,N)$ if $(Q,M)\leq(P,N)$ and $Q\leq_{G} P$.
\end{lemma}

\begin{proof}
This should be straightforward as we have already noticed in Lemma $5.2$ that if $M$ has vertex $Q$, then the terms of $X\otimes_{kH}M$ have, up to conjugacy, smaller vertices.
\end{proof}

\begin{proposition}
Suppose given $\leq$ on $I$ as before so that we can suppose $\leq$ respects the inclusions of vertices and let $\leq^{Q}$ denotes the local partial order induced by $\Br_{Q}$ on $\Ho(kN_{H}(Q)f_{Q}\text{-proj})$. Let $X$ a splendid $kGe$-$kHf$-complex such that for all $Q$ it induces perverse equivalences $\mathcal{D}^{b}(kN_{H}(Q)f_{Q}) \stackrel{\sim}{\rightarrow} \mathcal{D}^{b}(kN_{G}(Q)e_{Q})$ relative to $\leq^{Q}$. Then $\Ho(kHf\per)  \stackrel{\sim}{\rightarrow} \Ho(kGe\per)$ is a perverse equivalence relative to $\leq$.
\end{proposition}

\begin{proof}
We refer to our previous commutative diagram and use that if $P<Q$ then $(P,N)<(Q,M)$ so that the terms of $X\otimes_{kH}M$ with vertex strictly smaller than $Q$ belongs to $\mathcal{I}_{<\phi(Q,M)}$.
\end{proof}

We now have proved enough to state the following equivalence:

\begin{theorem}
Suppose given a partial order $\leq$ on $I$ that respects the inclusions of vertices and a locally compatible perversity datum $(\leq, p)$. Then the splendid equivalence $X\otimes_{kH}-: \Ho(kHf\per)  \stackrel{\sim}{\rightarrow} \Ho(kGe\per)$ is perverse relative to $\leq$ if and only if for all $Q\leq P$, it induces perverse equivalences $\mathcal{D}^{b}(kC_{H}(Q)f_{Q}) \stackrel{\sim}{\rightarrow} \mathcal{D}^{b}(kC_{G}(Q)e_{Q})$ relative to $\leq^{Q}$. 
\end{theorem}

\begin{remark}
The attentive reader might have spotted that we included in our previous result all subgroups of $P$... hence also the trivial subgroup! This amounts to include in our `local' data the perverse equivalence $\mathcal{D}^{b}(kHf) \stackrel{\sim}{\rightarrow} \mathcal{D}^{b}(kGe)$. In the following section, we shall try to understand how much of an obstruction on the trivial subgroup this condition really is. 
\end{remark}

At this stage, one might not be entirely satisfied with the previous result as our hope was to obtain a result in the flavour of Bouc-Rouquier's theorem (cf. Proposition $1.1$). It appears that the next thing to do would be to develop a reasonable notion of `perverse stable equivalence'.

\end{document}